\documentclass[a4paper, 11pt, reqno]{amsart}

\usepackage{amssymb}
\usepackage{amsfonts}
\usepackage{amsmath}
\usepackage[margin=2.5cm]{geometry}
\newtheorem{thm}{Theorem}[section]
\newtheorem{prop}[thm]{Proposition}
\newtheorem{lemma}[thm]{Lemma}
\newtheorem{cor}[thm]{Corollary}
\newtheorem{claim}[thm]{Claim}
\newtheorem{conj}[thm]{Conjecture} 
\newtheorem{ques}[thm]{Question}

\theoremstyle{definition}

\renewcommand{\epsilon}{\varepsilon}

\theoremstyle{remark}

\begin{document}

\title{Universality of graphs with few triangles and anti-triangles}

\author{Dan Hefetz}
\thanks{Research supported by EPSRC grant EP/K033379/1}
\address{School of Mathematics, University of Birmingham, 
Edgbaston, Birmingham, B15 2TT, United Kingdom}
\email{d.hefetz@bham.ac.uk, m.tyomkyn@bham.ac.uk}

\author{Mykhaylo Tyomkyn}
%\address{School of Mathematics, University of Birmingham, 
%Edgbaston, Birmingham, B15 2TT, United Kingdom}
%\email{m.tyomkyn@bham.ac.uk}

\begin{abstract}
We study \emph{3-random-like} graphs, that is, sequences of graphs in which the densities of triangles and anti-triangles converge to $1/8$. Since the random graph ${\mathcal G}_{n,1/2}$ is, in particular, 3-random-like, this can be viewed as a weak version of quasirandomness. We first show that $3$-random-like graphs are $4$-universal, that is, they contain induced copies of all $4$-vertex graphs. This settles a question of Linial and Morgenstern~\cite{Limo}. We then show that for larger subgraphs, $3$-random-like sequences demonstrate a completely different behaviour. We prove that for every graph $H$ on $n\geq R(10,10)$ vertices there exist $3$-random-like graphs without an induced copy of $H$. Moreover, we prove that for every $\ell$ there are $3$-random-like graphs which are $\ell$-universal but not $m$-universal when $m$ is sufficiently large compared to $\ell$.

\end{abstract}

\maketitle

\section{Introduction}

A graph is called $\ell$-universal if it contains every $\ell$-vertex graph as an induced subgraph. Universality is a well-studied graph property, for instance, the famous Erd\H{o}s-Hajnal conjecture \cite{EH} can be formulated in the following form. 

\begin{conj}
[Erd\H{o}s-Hajnal] For every integer $\ell$ there exists an $\epsilon > 0$ such that every $n$-vertex graph $G$ with no clique or independent set of size $n^\epsilon$ is $\ell$-universal.
\end{conj}

Recently Linial and Morgenstern \cite{Limo} asked a question of a similar flavour. Instead of forbidding large cliques and independent sets (anti-cliques) they asked, what happens if the graph $G$ contains only \emph{few} cliques and anticliques of a certain order $m$. The present paper addresses this question.

First, let us introduce some useful notation and terminology, most of which is standard (see e.g. \cite{Bela}). For a graph $G$ write $V(G)$ and $E(G)$ for its sets of vertices and edges, respectively. Let $|G| = |V(G)|$ denote the \emph{order} of $G$ and let $e(G) = |E(G)|$ denote its \emph{size}. The \emph{complement} of $G$ is denoted by $\overline{G}$. For a set $S \subseteq V(G)$ put $G[S]$ for the subgraph of $G$ induced on the set $S$. For a set $S \subseteq V(G)$ and a vertex $u \in V(G)$, let $N_G(u, S) = \{w \in S : uw \in E(G)\}$ denote the set of neighbours of $u$ in $S$ and let $d_G(u, S) = |N_G(u, S)|$ denote the \emph{degree} of $u$ into $S$. We abbreviate $N_G(u, V(G))$ to $N_G(u)$ and $d_G(u, V(G))$ to $d_G(u)$. The former is referred to as the \emph{neighbourhood} of $u$ in $G$ and the latter as its \emph{degree}. We use $d_G(u,v)$ to denote the \emph{co-degree} of $u$ and $v$, that is, $\left|N_G(u) \cap N_G(v)\right|$ and the somewhat less standard $d_G(u,-v)$ to denote $\left|N_G(u) \setminus N_G(v)\right|$. Often, when there is no risk of confusion, we omit the subscript $G$ from the notation above. 

For graphs $G$ and $H$, put $D_H(G)$ for the number of induced copies of $H$ in $G$ and $p_H(G)$ for the corresponding density: 

$$
p_H(G) = {\binom{n}{\left|H\right|}}^{-1}\cdot D_H(G) \,.
$$

The quantity $p_H(G)$ can be also interpreted as the probability that a randomly picked set of $|H|$ vertices of $G$ induces a copy of $H$. 

For $H=K_2$, a single edge, $D_H(G)$ is simply $e(G)$ and thus we write $p_e(G)$ for $p_{K_2}(G)$, the \emph{edge density} of $G$. For graphs of order $3$, since they are determined up to isomorphism by their size, we write $D_i(G)$ for $D_H(G)$ and $p_i(G)$ for $p_H(G)$, where $i = e(H)$. The vector $(p_0(G),\dots, p_3(G))$ is called the \emph{$3$-local profile} of $G$.

Let $\mathcal{G}=(G_k)_{k=1}^{\infty}$ be a sequence of graphs, where $G_k = (V_k, E_k)$ is of order $n_k := \left|V_k\right|$ and $n_k$ tends to infinity with $k$. If for some graph parameter $\lambda$ the limit $\lim_{k\rightarrow \infty} \lambda(G_k)$ exists, we denote it by $\lambda(\mathcal{G})$. A sequence $\mathcal{G}$ is said to be \emph{$\ell$-universal} if $G_k$ is $\ell$-universal for every sufficiently large $k$. 

Linial and Morgenstern proved in \cite{Limo} that there exists a constant $\rho = 0.159181\dots$ such that every $\mathcal{G}$ with $p_0(\mathcal G), p_3(\mathcal G) < \rho$ is $3$-universal and asked whether an analogous result holds for higher universalities.

\begin{ques} [\cite{Limo}] \label{qlimo}
Given $\ell\geq 4$, is there some $\epsilon > 0$ such that every graph sequence $\mathcal{G}$ with $p_0(\mathcal{G}), p_3(\mathcal{G}) < \frac{1}{8}+\epsilon$ is $\ell$-universal? 
\end{ques}

Note that our definition of \emph{$\ell$-universal} sequences is slightly different from the one given in \cite{Limo}. The latter required additionally that $p_{G_k}(H)$ be bounded away from $0$ for each $H$ of order $\ell$. However for our purposes (i.e. answering Question~\ref{qlimo}) these definitions are equivalent due to the induced graph removal lemma of Alon, Fischer, Krivelevich and Szegedy \cite{AFKS}.

It was pointed out by the second author that for every $\ell \geq 5$ the answer to Question~\ref{qlimo} is negative. Though his counterexample has already appeared in \cite{Limo}, for the sake of completeness we will repeat it in the next section of the present paper. 

This leaves $\ell=4$ as the only remaining open case of Question~\ref{qlimo}. Our first main result in this paper, Theorem \ref{thm4univ}, answers it in the affirmative, thereby settling Question~\ref{qlimo} in full.

Let us define a sequence of graphs $\mathcal{G}$ to be \emph{$t$-random-like}, or \emph{tRL} for brevity, if $p_{K_t}(\mathcal{G})=p_{\overline{K_t}}(\mathcal{G})=2^{-\binom{t}{2}}$. Our choice of terminology stems from the fact that such a sequence has approximately the same number of $t$-cliques and $t$-anticliques, that is, independent sets of size $t$, as the random graph $\mathcal{G}_{n,1/2}$. Note that for $\mathcal{G}$ to be 2RL it is sufficient to have $p_e(\mathcal{G})=1/2$. We will be mostly interested in 3RL sequences; in our terminology $\mathcal{G}$ is 3RL if and only if $p_0(\mathcal{G})=p_3(\mathcal{G})=1/8$. 

A standard diagonalisation argument shows that in order to answer Question \ref{qlimo} for $\ell=4$ affirmatively, it suffices to prove the following assertion.

\begin{thm}\label{thm4univ}
Every 3RL sequence is $4$-universal.
\end{thm}

Theorem \ref{thm4univ} is related to the \emph{quasirandomness} of graphs as well. This is a central notion in extremal and probabilistic graph theory. It was introduced by Thomason in \cite{Thomason2} and was extensively studied in many subsequent papers. In particular, it was proved by Chung, Graham and Wilson \cite{CGW} (see also \cite{alon:probabilistic} for more details) that if $p_H(\mathcal{G}) = p_H(\mathcal{G}_{n, 1/2})$ holds for every graph $H$ of order $4$, then the same equality holds for every graph $H$ of \emph{any} fixed size. In the terminology of \cite{CGW} this fact is denoted by $P_1(4)\Rightarrow P_1(s)$. On the other hand, it was pointed out in \cite{CGW} that the property $P_1(3)$, that is, containing the ``correct'' number of induced copies of every $3$-vertex graph, is not sufficient to ensure quasirandomness. As we shall see in Section \ref{secprelim}, $P_1(3)$ is in fact equivalent to 3RL. Thus, our results in this paper can be viewed as the study of $P_1(3)$. Under this viewpoint Theorem \ref{thm4univ} shows that, while 3RL graphs need not satisfy $P_1(4)$, they still must contain a positive density of every possible induced $4$-vertex graph. 

Having resolved Question~\ref{qlimo}, we know that 3RL implies $4$-universality, but is not enough to ensure $\ell$-universality for any larger $\ell$. A natural follow up question to ask is, whether there still exist infinite classes of graphs $H$ that must be contained in every 3RL sequence $\mathcal{G}$. Cliques, paths, cycles and stars are natural candidates for such classes. We shall answer this question in the negative by providing counterexamples for each of these classes. In fact, our second main result, Theorem \ref{thmoverkill} provides, perhaps surprisingly, a counterexample for \emph{any single} graph which is not too small. Throughout this paper $R(k,\ell)$ will stand, as usual, for the corresponding Ramsey number (see \cite{Bela} for more background details). 

\begin{thm}\label{thmoverkill}
For every graph $H$ of order at least $R(10,10)$ there exists a 3RL sequence $\mathcal{G}$, where no $G_k\in \mathcal{G}$ contains a copy of $H$ as an induced subgraph.
\end{thm}

According to \cite{SmallRam}, the best currently known bounds on $R(10,10)$ are $798\leq R(10,10)\leq 23556$ (the standard upper bound for Ramsey numbers yields $R(10,10)\leq\binom{10+10-2}{10-1}=48620$). 

Theorem \ref{thmoverkill} combined with Theorem \ref{thm4univ} and the induced graph removal lemma~\cite{AFKS} immediately give the following corollary.

\begin{cor}\label{cor4univ}
There exists an $\epsilon>0$ such that for every graph $H$ of order at least $R(10,10)$ there is a sequence $\mathcal{G}$, where no $G_k\in \mathcal{G}$ contains an induced copy of $H$, but $p_J(\mathcal{G})>\epsilon$ for every $4$-vertex graph $J$.  
\end{cor}

Theorem \ref{thmoverkill} and Corollary \ref{cor4univ} show that, for sufficiently large values of $\ell$, having either the ``correct'' densities of triangles and anti-triangles or positive densities of every $4$-vertex graph is far from being enough to ensure $\ell$-universality. This goes in stark contrast with $\mathcal{G}$ having the ``correct'' densities of \emph{all} induced $4$-vertex graphs, which implies that $\mathcal{G}$ is quasirandom and therefore $\ell$-universal for every $\ell$.
%, which among other things means that the number of induced copies of any fixed graph $H$ behaves random-like.

Having constructions of 3RL sequences which are only $\ell$-universal for very small values of $\ell$ on the one hand and the random graph $\mathcal{G}_{n,1/2}$ (which is $\ell$-universal for every fixed $\ell$) on the other hand, it is natural to ask, if for arbitrarily large $\ell$ there exists a 3RL sequence which is $\ell$-universal but not $f(\ell)$-universal for some function $f$. This would show that no fixed universality is sufficient to ensure all other universalities. Our third theorem shows that this is indeed the case in the following strong sense.

%Such a construction would bridge the gap between small and infinte universalities.

\begin{thm}\label{thmarbitrary}
For every $\ell$ there exists a 3RL sequence $\mathcal{G}_{\ell}$ such that $p_H(\mathcal{G_\ell})>0$ for every graph $H$ of order $2^\ell$, but $\mathcal{G}_\ell$ is not $24 \ell \cdot 2^\ell$-universal.
\end{thm}

The rest of this paper is organised as follows. In the next section we establish some basic properties of 3RL sequences and recall the construction of a 3RL sequence which is not $5$-universal. In Section $\ref{sec4univ}$ we prove our first main result, Theorem \ref{thm4univ}, by considering each `forbidden' $4$-vertex subgraph individually and applying different methods in different cases. In Section \ref{seclarge} we prove our second main result, Theorem \ref{thmoverkill}. This will be achieved through constructing a 3RL sequence in which no graph contains a clique of size $10$; we think that this construction is also of independent interest. In Section \ref{sechigh} we prove Theorem \ref{thmarbitrary} by adapting a construction of Chung, Graham and Wilson from their seminal paper on quasirandomness \cite{CGW}. Finally, in Section \ref{secdisc} we state a number of open questions and outline some possible extensions of our results. 

\section{Preliminaries}\label{secprelim}

Goodman's Theorem \cite{Goodman} gives a formula for the number of triangles and anti-triangles in a graph $G=(V,E)$ of order $n$:
\begin{equation}\label{eq: 1}
D_0(G) + D_3(G) = \frac{1}{2}\left[-\binom{n}{3}+\sum_{v \in V} \left[\binom{d_{G} (v)}{2}+ \binom{d_{\overline{G}}(v)}{2}\right]\right] \,.
\end{equation}

%where $d_{G}(v)$ and $d_{\overline{G}}(v)$ stand, as usual, for degrees of the vertex $v$ in $G$ and its complement $\overline{G}$. 
For densities this translates into
$$
p_0(G) + p_3(G) = \frac{\sum_{v \in V} \left[\binom{d_{G} (v)}{2}+ \binom{d_{\overline{G}} (v)}{2}\right]}{2\binom{n}{3}}-\frac{1}{2} \,.
$$ 
Since $d_G(v)+d_{\overline{G}}(v)=n-1$ for every $v$, and due to the convexity of binomial coefficients, the minimal value of $D_0+D_3$ is achieved whenever the degree of each vertex $v$ is as close to $n/2$ as possible, resulting in $p_0+p_3$ being asymptotically $1/4$. This is known as \emph{Goodman's bound}. Consequently, let us call $\mathcal{G}$ a \emph{Goodman sequence} if $(p_0+p_3)(\mathcal{G})=1/4$; note that we do not require the existence of the individual limits $p_0(\mathcal{G})$ and $p_3(\mathcal{G})$. Needless to say that 3RL sequences are Goodman. Applying a common abuse of terminology, we will talk about Goodman and 3RL \emph{graphs} referring to respective sequences of graphs. Notice that, since $D_H(G)=D_{\overline{H}}(\overline{G})$, a graph $G$ is Goodman (respectively 3RL) if and only if $\overline{G}$ is Goodman (respectively 3RL).

A vertex $v$ of a graph $G$ on $n$ vertices is said to be \emph{$\epsilon$-ordinary} if $\left|d_G(v)-n/2\right|<\epsilon n$ and \emph{$\epsilon$-exceptional} otherwise. Occasionally we will suppress the $\epsilon$ in the above notation if there is no ambiguity. The following fact is an immediate consequence of Goodman's Theorem. It asserts that a Goodman graph is essentially $n/2$-regular. 
%, that is all but few degrees in $G$ are close $n/2$. 

\begin{prop}\label{propgood}
For every $\epsilon > 0$ and every Goodman sequence $\mathcal{G}=(G_k)_{k=1}^{\infty}$ there exists an integer $k_0(\epsilon, \mathcal{G})$ such that for every $k \geq k_0$ at most $\epsilon n_k$ vertices of $G_k$ are $\epsilon$-exceptional.
\end{prop}

Note that, in particular, Goodman graphs are 2RL, thus ``Goodman'' can be considered an intermediate level between 2RL and 3RL. This was already pointed out by Chung, Graham and Wilson \cite{CGW} (``$P_1(3)\Rightarrow P_0 \Rightarrow P_1(2)$'' in their terminology; Corollary \ref{cor3rl} below states that $P_1(3)$ and 3RL are equivalent).   

\begin{proof}
Consider a graph $G$ of order $n$ in which at least $\epsilon n$ vertices are $\epsilon$-exceptional. Due to the convexity of binomial coefficients, each exceptional vertex $v$ contributes to the right hand side of \eqref{eq: 1} at least
\begin{eqnarray*}
\binom{d_{G}(v)}{2} + \binom{d_{\overline{G}}(v)}{2} 
&\geq& \binom{(1/2-\epsilon)n}{2} + \binom{(1/2+\epsilon)n}{2} + O(n) \\
&=& \left(1 + 4 \epsilon^2 \right) \left[\binom{n/2}{2} + \binom{n/2}{2} \right] + O(n) \,.
\end{eqnarray*}
If this happens $\epsilon n$ times, then $(D_0 + D_3)(G)$ exceeds its minimum possible value by at least $c \epsilon^3 n^3$ for some constant $c > 0$. Therefore $(p_0 + p_3)(G) > 1/4 + c' \epsilon^3 + o(1)$ for some absolute constant $c' > 0$. This can only happen finitely many times in a Goodman sequence. 
\end{proof}

Conversely, it is easy to see that every $\mathcal{G}$ satisfying the above is Goodman. In other words, Proposition \ref{propgood} gives an alternative characterisation of Goodman sequences. 

%This allows us to give a very simple proof of the $3$-universality of 3RL sequences, taking a different approach to \cite{Limo}, where the authors minimized $p_0+p_3$ assuming $p_0=p_3$. 

%\begin{thm}\label{3univ}
%(\cite{Limo}) Every 3RL sequence is $3$-universal.
%\end{thm}

%\begin{proof}
%By symmetry it suffices to obtain a contradiction if $G$ does not contain a graph on $3$ vertices and $2$ edges. Observe that in that case all connected components of $G$ must be complete graphs. By Proposition \ref{propgood}, for large $n$ all but at most $\epsilon n$ vertices are ordinary. Hence the size of each component must be either within $\epsilon n$ of $n/2$ (let us call them \emph{large} components) or less than $\epsilon n$ (\emph{small} components). Since the total size of all small components is at most $\epsilon n$, the number of large components is $2$. With $\epsilon\rightarrow 0$ we obtain two large components of size $n/2+o(n)$ each and some small components of total size $o(n)$. However, this would produce a total of $\frac{1}{4}\binom{n}{3}+o(n^3)$ triangles in the large components, resulting in $p_3\geq 1/4$, a contradiction. 
%\end{proof}
 
The next lemma and its corollary can be viewed as a strengthening of the $3$-universality result from \cite{Limo} (although, unlike Linial and Morgenstern, we do not optimise the error term $\varepsilon$). It provides additional information about the $3$-local profile of Goodman graphs, asserting that it is determined completely by $p_0$ (and, equivalently, by $p_3$).

\begin{lemma}\label{lemgood}
If $\mathcal{G}$ is Goodman then $(p_1-3p_3)(\mathcal{G})=(p_2-3p_0)(\mathcal{G})=0$.
\end{lemma}

\begin{proof}
Counting vertex-edge pairs $(v,e)$ of $G \in \mathcal{G}$, where $v \notin e$ in two different ways, we obtain 
\begin{equation} \label{eq::doubleCount1}
3D_3(G)+2D_2(G)+D_1(G) = (n-2)e(G) \,.
\end{equation}
%where $e(G)$ stands for, as usual, the number of edges in $G$. 
Similarly, counting such vertex-edge pairs in $\overline{G}$, we obtain 
\begin{equation} \label{eq::doubleCount2}
3D_0(G)+2D_1(G)+D_2(G) = 3D_3(\overline{G})+2D_2(\overline{G})+D_1(\overline{G}) = (n-2)e(\overline{G}) \,.
\end{equation}
Since, by Proposition \ref{propgood}, $|e(H) - e(\overline{H})| = o(|H|^2)$ holds for every Goodman graph $H$, we obtain asymptotic equality between the left hand sides of~\eqref{eq::doubleCount1} and~\eqref{eq::doubleCount2}. Passing to densities, this translates into
$$
3p_3+p_2=3p_0+p_1=\frac{1}{2}(3p_3+p_2+p_1+3p_0)=\frac{1}{2}\left[(p_0+p_1+p_2+p_3)+2(p_0+p_3)\right]=\frac{3}{4} \,,
$$
hence 
$$ 
p_1=\frac{3}{4}-3p_0=3p_3 \,.
$$
Similarly, $p_2=3p_0$.
\end {proof}

As an immediate consequence, we determine the $3$-local profile of 3RL graphs.

\begin{cor}\label{cor3rl}
If $\mathcal{G}$ is 3RL then $p_1(\mathcal{G})=p_2(\mathcal{G})=3/8$.
\end{cor}

In other words, the $3$-local profile of a 3RL graph mirrors that of the random graph $\mathcal{G}_{n,1/2}$, justifying our choice of terminology. 

The following construction from \cite{Limo} %, communicated by the second author, 
is known as the \emph{iterated blow-up} (see e.g. \cite{Victor}) and demonstrates that 3RL graphs need not be $5$-universal. Let $G_1 \cong C_5$ be a $5$-cycle. Given $G_k$, construct $G_{k+1}$ as follows. Take a $5$-\emph{blow-up} of $G_k$ (that is, replace every vertex $v$ of $G_k$ by $5$ new vertices $v_1,\dots, v_5$ and draw an edge between $u_i$ and $v_j$ if and only if there was an edge between $u$ and $v$), and add a $5$-cycle within each set $v_1,\dots,v_5$. Alternatively, $G_{k+1}$ can be constructed by taking a $5^k$-blow-up of $C_5$ and adding a copy of $G_k$ on each partition class. It is not hard to check that no $G_k$ contains an induced path on $5$ vertices. %(which in turn implies that $G_k$ contains no path of length greater than $4$ and no cycle ).
In order to see that $\mathcal{G}$ is 3RL one can either calculate the densities directly (as in~\cite{Limo}) or observe that $\mathcal{G}$ is a sequence of self-complementary $\left\lfloor n/2\right\rfloor$-regular graphs, which by Goodman's Theorem yields $p_0(\mathcal{G}) = p_3(\mathcal{G}) = 1/8$.

Note that this construction also shows that for every $\ell \geq 5$ and every $r \geq 6$ there exist arbitrarily large 3RL graphs which do not contain the path $P_{\ell}$ of length $\ell-1$ or the cycle $C_r$ as induced subgraphs. This is because any graph which contains an induced $P_\ell$ for some $\ell \geq 5$ or an induced $C_r$ for some $r \geq 6$ contains an induced $P_5$. The case of the $5$-cycle remains open.

\section{Proof of Theorem \ref{thm4univ}}\label{sec4univ}

We have to show that a sufficiently large 3RL graph $G$ contains each graph of order $4$ as an induced subgraph. Note that, in contrast to Corollary~\ref{cor3rl}, we cannot expect the density of $H$ in $G$ to be random-like for every graph $H$ on 4 vertices. Indeed, it is well-known (see e.g. Theorem 9.3.1 in~\cite{alon:probabilistic}) that such graphs are quasirandom and thus, in particular, $\ell$-universal for any fixed $\ell$. 

Since $G$ is 3RL if and only if $\overline{G}$ is, and the induced subgraphs of the latter are precisely the complements of induced subgraphs of the former, it suffices to split all $4$-vertex graphs into complementary pairs (the graph $P_4$, the path of length three, is self complementary) and prove containment for one graph $H$ from each pair. Thus we need only consider the following $6$ cases:

\begin{itemize}
\item $H=K_4$, the complete $4$-vertex graph
\item $H=K_4^-$, the complete graph with one edge missing
\item $H=C_4$, the $4$-cycle
\item $H=T^+$, a triangle with a pendant edge
\item $H=K_{1,3}$, the star (also known as the claw)
\item $H=P_4$, the path of length 3 
\end{itemize}

While the graphs above are listed in order of decreasing number of edges, we will consider them in a different order, starting from what we believe is the simplest case and finishing with the most difficult. In each of the cases the containment of $H$ is proved by contradiction, assuming initially that $\mathcal{G}$ is 3RL and $H$-free (remember that we are always looking for an \emph{induced} copy of $H$). 

\subsection*{Case 1: $H =T^+$}

It follows by Proposition \ref{propgood} that, for every $\epsilon > 0$ and sufficiently large $n$, if $G \in \mathcal{G}$ is a graph on $n$ vertices, then it contains at most $\epsilon n$ exceptional vertices. The set of all exceptional vertices of $G$ can intersect at most $\epsilon n^3$ triangles. Since $G$ is 3RL, it contains $(1/48+o(1))n^3$ triangles and so, for sufficiently large $n$, there must exist $\epsilon$-ordinary vertices $u$, $v$ and $w$ which form a triangle $T$ in $G$. 

Since $G$ is $T^+$-free, for any $x \in V(G) \setminus \{u,v,w\}$ we must have $d_G(x, T) \in \{0,2,3\}$. We partition the vertices of $V(G) \setminus \{u,v,w\}$ into two sets $X = \{x \in V(G) \setminus \{u,v,w\} : d_G(x, T) = 0\}$ and $Y = \{x \in V(G) \setminus \{u,v,w\} : d_G(x, T) \in \{2,3\}\}$. Since $u$, $v$ and $w$ are ordinary, on average, a vertex $x \in V(G) \setminus \{u,v,w\}$ will have $3/2 + o(1)$ neighbours in $T$. Therefore, we must have $n/4 - o(n) \leq |X| \leq n/2 + o(n)$ and $n/2 - o(n) \leq |Y| \leq 3n/4 + o(n)$. Let $x \in X$ and $y \in Y$ be arbitrary vertices. Assume without loss of generality that $\{u,v\} \subseteq N_G(y, T)$. We conclude that $x$ and $y$ are not adjacent in $G$ as otherwise  the vertices $x, y, u$ and $v$ would form an induced copy of $T^+$ in $G$. 

Since $x$ and $y$ were arbitrary, it follows that there are no edges of $G$ between $X$ and $Y$. Since $|X| \geq n/4 - o(n)$ and at most $\epsilon n$ vertices of $G$ are exceptional, there exists some ordinary $x \in X$. Because of $N_G(x) \subseteq X$, it follows that $|X| = n/2 + o(n)$. Finally, since all but at most $\epsilon n$ vertices of $X$ are ordinary and each of them has degree $n/2 + o(n)$ in $X$, we conclude that $e(G[X]) \geq \binom{n/2}{2} - o(n^2)$. 

A similar argument shows that $|Y| = n/2 + o(n)$ and that $e(G[Y]) \geq \binom{n/2}{2} - o(n^2)$. Counting anti-triangles in $G$, it follows that $D_0(G) = o(n^3)$ and thus $p_0(G) = 0$, contrary to our assumption that $G$ is 3RL.        

\subsection*{Case 2: $H = K_4^-$}

As in Case 1, consider a triangle $T=\left\{u,v,w\right\}$ where $u$, $v$ and $w$ are ordinary vertices. Since $G$ is $K_4^-$-free, for any $x \in V(G) \setminus \{u,v,w\}$ we must have $d_G(x, T) \in \{0,1,3\}$. We partition the vertices of $V(G)$ into two sets $X = \{x \in V(G) \setminus \{u,v,w\} : d_G(x, T) \in \{0,1\}\}$ and $Y = V(G) \setminus X$. Since $u$, $v$ and $w$ are ordinary, on average a vertex $x \in V(G) \setminus \{u,v,w\}$ will have $3/2 + o(1)$ neighbours in $T$. Therefore, we must have $n/2 - o(n) \leq |X| \leq 3n/4 + o(n)$ and $n/4 - o(n) \leq |Y| \leq n/2 + o(n)$. Considering $u$, $v$ and arbitrary $x, y \in Y \setminus \{u,v\}$, we deduce that $x$ and $y$ are adjacent in $G$, for otherwise $u$ ,$v$, $x$ and $y$ would form an induced copy of $K_4^-$ in $G$. It follows that $G[Y]$ is a clique. 

Let $z \in X$ be an arbitrary vertex and assume without loss of generality that $\{z, u\} \notin E(G)$. Let $y_1, y_2 \in Y \setminus \{u\}$ be arbitrary vertices. If $\{z, y_1\} \in E(G)$ and $\{z, y_2\} \in E(G)$, then $G[\{z, y_1, y_2, u\}]$ is an induced copy of $K_4^-$ in $G$. It follows that $d_G(x, Y) \leq 1$ for every $x \in X$. By Proposition~\ref{propgood} $G$ is essentially $n/2$-regular and thus we must have $n/2 - o(n) \leq |X|, |Y| \leq n/2 + o(n)$ and $e(G[X]) \geq \binom{n/2}{2} - o(n^2)$. Similarly to Case 1, it follows that $p_0(G) = 0$, contrary to our assumption that $G$ is 3RL.     

\subsection*{Case 3: $H = K_4$}

Let $\mathcal{G'}$ be any Goodman (not necessarily 3RL) sequence of $K_4$-free graphs. Observe that $G'$ being $K_4$-free is equivalent to $N_{G'}(v)$ being triangle-free for every $v \in V(G')$. Therefore, by Mantel's Theorem, $e(G'[N(v)]) \leq d(v)^2/4$ for every $v \in V(G')$. Since $\mathcal{G'}$ is Goodman, it follows by Proposition \ref{propgood} that the neighbourhoods of all but $o(n)$ vertices of $G'\in \mathcal{G'}$, span at most $(n/2+o(n))^2/4 = n^2/16+o(n^2)$ edges. Since $e(G'[N_{G'}(v)])$ is precisely the number of triangles of $G'$ that include $v$, a double counting of the edges in all neighbourhoods shows that 
$$
3D_3 = \sum_{v \in V(G')} e(G'[N(v)]) \leq \sum_{v \in V(G')} d(v)^2/4 = n \cdot \frac{n^2}{16} + o(n^3) =\left(\frac{3}{8} + o(1)\right) \binom{n}{3} \,.
$$
Hence, for any $K_4$-free Goodman sequence $\mathcal{G'}$ the value $p_3(\mathcal{G'})$, if it exists, is at most $1/8$, with equality attained only when $e(G'[N(v)]) = (1 - o(1)) d(v)^2/4$ holds for all but $o(n)$ vertices $v \in V(G')$. Conversely, $\mathcal{G}$ being 3RL implies that equality must be attained, whence we conclude that $e(G[N(v)]) = (1 - o(1)) d(v)^2/4$ holds for all but $o(n)$ vertices $v \in V(G)$.

Structural information on such `nearly extremal' graphs is provided by the Erd\H{o}s-Simonovits stability Theorem \cite{Simonovits} which, in this particular case and combined with the above, asserts that there exists a set $U \subseteq V(G)$ of order $(1 - o(1)) n$ such that for every $v \in U$ we have $d_G(v) = n/2 + o(n)$ and the neighbourhood $N(v)$ admits a bipartition into parts $N_1(v)$ and $N_2(v)$ such that $|N_1(v)|, |N_2(v)| = n/4 + o(n)$, there are $o(n^2)$ edges within each partition class and $(1-o(1)) n^2/16$ edges between the two classes. 

Let $v \in U$ be an arbitrary vertex. It follows by the above that there exists a vertex $u \in U \cap N_1(v)$ such that $d_G(u, N_2(v)) = n/4 + o(n)$ and $d_G(u, N_1(v)) = o(n)$ (recall that almost every vertex has these properties). Let $B = N_G(u) \setminus N_G(v)$; note that $|B| = n/4 + o(n)$. Since $u \in U$, its neighbourhood $N_G(u)$ must induce an essentially complete bipartite graph with both parts of order $n/4 + o(n)$. Since $N_2(u) := N_2(v) \cap N_G(u)$ is of order $n/4 + o(n)$ and contains $o(n^2)$ edges, up to $o(n)$ changes, the only way to achieve this is by taking the bipartition to be $N_G(u) = B \cup N_2(u)$. Let $w \in U \cap N_2(u)$ be a vertex such that $d_G(w, N_1(v)) = n/4 + o(n)$ and $d_G(w, B) = n/4 + o(n)$; by the above, almost every vertex of $U \cap N_2(u)$ has these properties. Since $w \in U$, its neighbourhood $N_G(w)$ must induce an essentially complete bipartite graph with both parts of order $n/4 + o(n)$. Up to $o(n)$ changes, the only way to achieve this is by taking the bipartition to be $N_G(w) = B \cup N_1(v)$. We conclude that the sets $N_1(v)$, $N_2(u)$ and $B$ are of size $n/4 + o(n)$ each and $G[N_1(v) \cup N_2(u) \cup B]$ is essentially an $n/2$-regular tripartite graph. 

Let $X = N_1(v) \cup N_2(u) \cup B$ and let $Y = U \setminus X$. On the one hand, $|Y| = n/4 + o(n)$ and the degree of every vertex in $Y$ is $n/2 + o(n)$ entailing that there are $\Omega(n^2)$ edges between $X$ and $Y$. On the other hand, all but $o(n)$ vertices of $X$ have degree $n/2 + o(n)$ in $G$ and in $G[X]$ entailing that there are $o(n^2)$ edges between $X$ and $Y$. This is clearly a contradiction.          

\subsection*{Case 4: $H = K_{1,3}$}
Let $\mathcal{G'}$ be a Goodman sequence of $K_{1,3}$-free graphs. Note that $G'$ being $K_{1,3}$-free is equivalent to $N_{G'}(v)$ being anti-triangle-free for every $v \in V(G')$, that is, the non-edges in $N_{G'}(v)$ must not form a triangle. Similarly to Case 3, it follows from Mantel's Theorem and Proposition \ref{propgood} that the neighbourhoods of all but $o(n)$ vertices of $G'$, span at most $(n/2+o(n))^2/4 = n^2/16+o(n^2)$ non-edges. Double counting of the non-edges in all neighbourhoods shows that   
$$
D_2 = \sum_{v \in V(G')} e(\overline{G'}[N_{G'}(v)]) = n \cdot \left(n^2/16 + o(n^2)\right) = \left(\frac{3}{8} + o(1)\right) \binom{n}{3} \,.
$$
Hence, for any Goodman sequence $\mathcal{G'}$, the value $p_2(\mathcal{G'})$, if it exists, is at most $3/8$, where by the Erd\H{o}s-Simonovits stability Theorem, equality is attained only when almost all neighbourhoods are close to being disjoint unions of two complete graphs of order $n/4$ each. Since $G$ is 3RL, it follows by Corollary~\ref{cor3rl} that this must indeed be the case. 

Let $U \subseteq V(G)$ be a set of order $(1 - o(1)) n$ such that for every $v \in U$ we have $d_G(v) = n/2 + o(n)$ and the neighbourhood $N(v)$ admits a bipartition into parts $N_1(v)$ and $N_2(v)$ such that $|N_1(v)|, |N_2(v)| = n/4 + o(n)$, there are $(1 - o(1)) n^2/32$ edges within each partition class and $o(n^2)$ edges between the two classes. 

Let $v \in U$ be an arbitrary vertex. It follows by the above that there exists a vertex $u \in A(u) := U \cap N_1(v)$ such that $|A(u)| = n/4 + o(n)$, $d_G(u, A(u)) = n/4 + o(n)$ and $d_G(u, N_2(v)) = o(n)$. Let $B = N_G(u) \setminus N_G(v)$; note that $|B| = n/4 + o(n)$. Since $u \in U$, its neighbourhood $N_G(u)$ must be close to a union of two complete graphs of order $n/4$ each. Since $G[A(u)]$ is essentially a complete graph on $n/4 + o(n)$ vertices, it follows that $G[B]$ is essentially a complete graph on $n/4 + o(n)$ vertices as well. Moreover, there are $o(n^2)$ edges of $G$ between $A(u)$ and $N_2(v) \cup B$.

Let $X = U \setminus (A(u) \cup N_2(v) \cup B)$; note that $|X| = n/4 + o(n)$. Since $d_G(w) = n/2 + o(n)$ holds for every $w \in A(u)$, it follows that $\{x,y\} \in E(G)$ for all but $o(n^2)$ pairs $(x,y) \in A(u) \times X$. Let $z \in A(u) \setminus \{u\}$ be an arbitrary vertex. Up to $o(n)$ vertices, its neighbourhood is $A(u) \cup X$ and so is far from being the disjoint union of two cliques of order $n/4$ each, contrary to the definition of $U$.         

\subsection*{Case 5: $H = P_4$}
For graphs with no induced $P_4$, also known as \emph{cographs}, we have the following structural characterisation due to Seinsche \cite{Seinsche}: if $G$ is induced $P_4$-free then either $G$ or $\overline{G}$ is disconnected (the other one is thereby forced to be connected). Let $W \subseteq V(G)$ be an arbitrary set of order at least 2. Clearly $G[W]$ is induced $P_4$-free and so, by the above characterisation, either $G[W]$ or $\overline{G}[W]$ is disconnected (note that it might be $G[W]$ for certain $W \subseteq V(G)$ and $\overline{G}[W]$ for others). 

Seinsche's characterisation allows us to construct a sequence $\mathcal{P}_0, \mathcal{P}_1, \dots$ of partitions of $V(G)$ as follows. $\mathcal{P}_0 = \left\{V(G)\right\}$ and, for every $i \geq 0$, $\mathcal{P}_{i+1}$ is obtained through partitioning each $W \in \mathcal{P}_i$ with $\left|W\right|\geq 2$ into the connected components of either $G[W]$ or $\overline{G}[W]$, depending which of the two is disconnected. For every $i \geq 0$, let $V_i$ denote a largest set in $\mathcal{P}_i$. For an arbitrarily small $\varepsilon > 0$ and sufficiently large $n$ let $j \geq 0$ denote the smallest index for which $|V_{j+1}| < (1 - \varepsilon) n$; clearly such an index $j$ must exist. Since $G$ is Goodman, it follows by Proposition~\ref{propgood} that at most $\epsilon n$ vertices of $G$ are $\epsilon$-exceptional. Since $|V_{j+1}| < (1 - \varepsilon) n$, the only way to ensure that there will not be too many exceptional vertices in $G$ is to split $V_j$ in $\mathcal{P}_{j+1}$ into two sets $W_1$ and $W_2$ of size $n/2 - 2 \epsilon n \leq |W_1|, |W_2| \leq n/2 + 2 \epsilon n$, and possibly some additional small sets. Indeed, otherwise every vertex of $V(G) \setminus V_{j+1}$ would be exceptional. Assume first that $G[V_j]$ is disconnected. Then there are at most $3\epsilon n^2$ pairs $\{x,y\} \subseteq W_1$ and at most $3\epsilon n^2$ pairs $\{x,y\} \subseteq W_2$ which are not adjacent in $G$. It follows that $D_0(G) \leq c \epsilon n^3$ for some absolute constant $c$ and thus $p_0(G) = 0$. Similarly, if $\overline{G}[V_j]$ is disconnected then $p_3(G) = 0$. This contradicts our assumption that $G$ is 3RL. 

\subsection*{Case 6: $H = C_4$}

Consider the expression 
$$
\sum_{\left\{u,v\right\}\in E(\overline{G})}\binom{d_G(u,v)}{2} \,.
$$ 
On the one hand, it counts $D_{K_4^-}(G)+2D_{C_4}(G)$, and since, by assumption, $D_{C_4}=0$, it must equal $D_{K_4^-}$. Now, since by Corollary \ref{cor3rl}

$$
\sum_{\left\{u,v\right\}\in E(\overline{G})}d_G(u,v)=D_2(G)=\frac{3}{8}\binom{n}{3}+o(n^3)=\frac{n^3}{16}+o(n^3) \,, 
$$

using the convexity of binomial coefficients we obtain

%\begin{equation}\label{eq: 2}
%\begin{split}
%D_{K_4^-}= \sum_{\left\{u,v\right\}\in E(\overline{G})}\binom{d_G(u,v)}{2} \geq e(\overline{G})\binom{\frac{1}{e(\overline{G})}\sum_{\left\{u,v\right\}\in E(\overline{G})}d_G(u,v)}{2}\\
%&=\frac{n^2}{4}\binom{n/4}{2}+o(n^4)=\frac{n^4}{128}+o(n^4).
%\end{split}
%\end{equation}

\begin{align}\label{eq: 2}
D_{K_4^-}&= \sum_{\left\{u,v\right\}\in E(\overline{G})}\binom{d_G(u,v)}{2} \geq e(\overline{G})\binom{\frac{1}{e(\overline{G})}\sum_{\left\{u,v\right\}\in E(\overline{G})}d_G(u,v)}{2}\\
&=\frac{n^2}{4}\binom{n/4}{2}+o(n^4)=\frac{n^4}{128}+o(n^4) \,. \notag
\end{align}

Thus 

\begin{equation}\label{eq: 2a}
p_{K_4^-}(G)\geq 3/16 \,.
\end{equation}

Now consider the expression 
$$
\sum_{\left\{u,v\right\}\in E(G)}d_G(u,-v)\cdot d_G(v,-u) \,.
$$
On the one hand, it counts $D_{P_4}+4D_{C_4}$, which by assumption equals $D_{P_4}$. On the other hand, since $G$ is Goodman, by Proposition \ref{propgood} we have $d(u)=d(v)+o(n)$ for all but at most $o(n^2)$ pairs of vertices, hence, $d(u,-v)=d(v,-u)+o(n)$ for almost all pairs. As a result, we obtain

$$
\sum_{\left\{u,v\right\}\in E(G)} d(u,-v)\cdot d(v,-u) = \frac{1}{4} \sum_{\left\{u,v\right\} \in E(G)} \left(d(u,-v) + d(v,-u)\right)^2 + o(n^4) \,.
$$

Now, since 
$$
\sum_{\left\{u,v\right\}\in E(G)}\left(d(u,-v)+d(v,-u)\right)=2D_2(G)=\frac{3}{4}\binom{n}{3}+o(n^3)=\frac{n^3}{8}+o(n^3) \,,
$$

the Cauchy-Schwarz inequality yields
\begin{align}\label{eq: 3}
D_{P_4} &= \frac{1}{4}\sum_{\left\{u,v\right\}\in E(G)}\left(d(u,-v)+d(v,-u)\right)^2+o(n^4) \\
&\geq \frac{1}{4} \cdot e(G)\left[\frac{1}{e(G)}\sum_{\left\{u,v\right\}\in E(G)}\left(d(u,-v)+d(v,-u)\right)\right]^2 +o(n^4) \notag \\
&= \frac{1}{4} \cdot \frac{n^2}{4} \cdot \left(\frac{n}{2}\right)^2+o(n^4) \notag \\ 
&= \frac{n^4}{64}+o(n^4) \,. \notag
\end{align}

Thus 

\begin{equation}\label{eq: 3a}
p_{P_4}(G) \geq 3/8 \,.
\end{equation}

Finally, double counting pairs of edges in $G$ not sharing a vertex, we obtain
$$
\frac{e(G)^2}{2}+o(n^4)=\binom{n^2/4}{2}+o(n^4)=D_{2K_2}+D_{P_4}+D_{T^+}+2D_{C_4}+2D_{K_4^-}+3D_{K_4} \,,
$$
where $2K_2$ is the complement of $C_4$. Thus 
$$
p_{2K_2}+p_{P_4}+p_{T^+}+2p_{K_4^-}+3p_{K_4} = \frac{3}{4} \,.
$$

Since from \eqref{eq: 2a} and \eqref{eq: 3a} we know that $2p_{K_4^-}(G)+p_{P_4}(G)\geq 3/4$, we deduce that $p_{T^+}(G)=p_{K_4}(G)=0$, $p_{K_4^-}(G)=3/16$, and $p_{P_4}(G) = 3/8$. The last two identities can only hold if in \eqref{eq: 2} and \eqref{eq: 3} we have equality up to $o(n^4)$. The former would imply that $d(u,v)$ must be close to its average $n/4 + o(n)$ for all but $o(n^2)$ pairs $\left\{u,v\right\}\in E(\overline{G})$. Similarly, an equality up to $o(n^4)$ in \eqref{eq: 3} implies that $d(u,-v)=d(v,-u)+o(n)=n/4+o(n)$ for almost every pair $\left\{u,v\right\}\in E(G)$, which, given Proposition \ref{propgood}, also means that $d(u,v)=n/4+o(n)$ for every such pair. In total, we obtain that all but $o(n^2)$ pairs of vertices $\left\{u,v\right\}$ have $n/4+o(n)$ joint neighbours, and therefore
\begin{equation} \label{eq::quasirandom}
\sum_{u,v \in V} \left|d(u,v)-\frac{n}{4}\right| = o(n^3) \,. 
\end{equation}
%, that is,  $d(u,v)=d(u',v')+o(n)=n/4+o(n)$; in the language of Probability, this is an instance of Chebyshev's inequality

However, equation~\eqref{eq::quasirandom} is one of several equivalent definitions of a quasirandom graph (see e.g. Theorem 9.3.1 in~\cite{alon:probabilistic}) and thus all induced densities in $G$ are random-like. In particular, contrary to our assumption, $G$ cannot be induced $C_4$-free.
%However, since we know from previous cases that $G$ must contain each of them, the induced graph removal lemma implies that $p_{K_4}(G)>0$ and $p_{T^+}(G)>0$, a contradiction.

\medskip
\noindent
With contradiction obtained for each $4$-vertex graph $H$, the proof of Theorem \ref{thm4univ} is concluded.

\section{Large induced subgraphs}\label{seclarge}

Our aim in this section is to prove Theorem \ref{thmoverkill}. The main ingredient of our proof will be a construction of a $3RL$ sequence with no cliques of order at least 10.  

Given any graph $G$ of order $n$, we construct an $(n-1)$-regular graph $H = f(G)$ of order $2n$ as follows. Let $G_1 = (V_1, E_1)$ and $G_2 = (V_2, E_2)$ be two disjoint copies of $G = (V,E)$, where $V = \{u_1, \ldots, u_n\}$, $V_1 = \{v_1, \ldots, v_n\}$ and $V_2 = \{v'_1, \ldots, v'_n\}$. Set $V(H) = V_1 \cup V_2$ and $E(H) = E_1 \cup E_2 \cup F$, where $F = \left\{\{v_i, v'_j\} \colon 1 \leq i \neq j \leq n \textrm{ and } \left\{u_i, u_j \right\} \notin E \right\}$. That is, we take two identical copies of $G$ and connect two distinct vertices, one from each copy, by an edge of $H$ if and only if they are \emph{not} adjacent in $G$. 

The above construction is very similar to the \emph{tensor product} of $G$ and $K_2$; the sole difference is that we exclude the ``vertical'' edges, that is, edges between $v_i$ and $v'_i$. Tensor products of graphs were first defined by Thomason in~\cite{Thomason}. See Section~\ref{secdisc} for more details. 

Note that for any sequence $\mathcal{G} = (G_k)_{k=1}^\infty$, the corresponding sequence $f(\mathcal{G}) = (f(G_k))_{k=1}^\infty$ is automatically Goodman. The next question to ask is, under what conditions is $f(\mathcal{G})$ 3RL.

\begin{lemma} \label{lem::fis3RL}
$\mathcal{H} = f(\mathcal{G})$ is 3RL if and only if $(p_0 + p_2)(\mathcal{G}) = (p_1 + p_3)(\mathcal{G}) = 1/2$. 
\end{lemma}

\begin{proof}
For every $0 \leq i \leq 3$, let $T_i$ denote the number of triangles of $H$ with exactly $i$ vertices in $V_1$. It is evident that $D_3(H) = T_0 + T_1 + T_2 + T_3 = 2(T_0 + T_1)$. Clearly $T_0 = D_3(G)$. Moreover, every triangle with exactly one vertex in $V_1$ corresponds to three vertices which induce precisely one edge in $G$ and thus $T_1 = D_1(G)$. It follows that $D_3(H) = 2(D_3(G) + D_1(G))$ and thus     
$$
p_3(H) = \frac{1}{4} \left[p_1(G) + p_3(G)\right] + o(1) \,.
$$
We conclude that $p_3(H) = 1/8$ if and only if $(p_1 + p_3)(G) = 1/2$. An analogous argument shows that $p_0(H) = 1/8$ if and only if $(p_0 + p_2)(G) = 1/2$. 
\end{proof}

Given Lemma~\ref{lem::fis3RL}, our aim is to construct a sequence $\mathcal{G}$ with $p_1(\mathcal{G}) + p_3(\mathcal{G}) = 1/2$ such that $f(\mathcal{G})$ does not contain a clique of some fixed size. Given a positive integer $k$ and a real number $r \geq 2$ that might depend on $k$, let $G_k^r = (V_k, E_k)$, where $V_k = \{0, 1, \ldots, k-1\}$ and $E_k = \{\{i,j\} \colon i - j \mod k > k/r  \textrm{ and } j - i \mod k > k/r\}$; let $\mathcal{G} = (G_k^r)_{k=1}^\infty$. Since any $\lceil r \rceil$ vertices of $V_k$ must contain two whose distance in the cyclic group $C_k$ is at most $k/r$, the largest clique of $G_k^r$ is of order at most $\left \lceil r \right \rceil - 1$. We claim that a similar statement holds in $H$.  

\begin{claim}
For every $r \geq 5$ and sufficiently large $k$, the largest clique in $H = f(G_k^r)$ is of order at most $\left \lceil r \right \rceil - 1$.
\end{claim}

\begin{proof}
For the sake of clarity of presentation let us assume that $r$ is an integer. Suppose for a contradiction that $\phi$ is an embedding of $K_r$ in $H$. Let $J$ denote the resulting copy; let $U_1 = V_1 \cap V(J)$ and $U_2 = V_2 \cap V(J)$. Since, as observed above, $G_k^r$ does not contain $K_r$ as a subgraph, it follows that $U_1 \neq \emptyset$ and $U_2 \neq \emptyset$. Let $t \geq s \geq 1$ be integers such that $s = |U_1|$ and $t = |U_2|$; note that $t \geq \left\lceil r/2\right\rceil \geq 3$. Since every $u \in U_1$ and every $v \in U_2$ are joined by an edge of $H$ and yet $\{v_i, v'_i\} \notin E(H)$ for every $1 \leq i \leq n$, it follows that $\overline{G_k^r}$ contains the complete bipartite graph $K_{s,t}$ as an induced subgraph. Let $u \in U_1$ and let $w_1$, $w_2$ and $w_3$ be three of its neighbours in $U_2$. These four vertices correspond to four vertices $i$ and $j_1 < j_2 < j_3$ of $V_k$ such that $j_1, j_2$ and $j_3$ are pairwise far in $C_k$ but $i$ is close to all of them. This is clearly impossible. 
\end{proof}       

It remains to find a value of $r$ for which $(p_1 + p_3)f(\mathcal{G}) = 1/2$, where $\mathcal{G} = (G_k^r)_{k=1}^\infty$. Observe that for $r > k$ the graph $G_k^r$ is a complete graph, whereas for $r = 2$ the graph $G_k^r$ is empty. Hence there must exist a real number $r$ for which $(p_1 + p_3)(G_k^r) = 1/2+o(1)$. A straightforward calculation shows that $p_3(G_k^r) = \left(\frac{r-3}{r}\right)^2 + o(1)$ and $p_1(G_k^r)= \frac{3}{r^2} + o(1)$, whence the desired value of $r$ is achieved at $2\sqrt{3} + 6 \approx 9.46$. For this value of $r$ the sequence $\mathcal{H} = f(\mathcal G)$ is 3RL but does not contain a clique of order $10$.

\medskip

Theorem~\ref{thmoverkill} is a simple corollary of the aforementioned result for cliques of order 10. Before showing this, let us remark that, for every $r \geq 10$, there exists a 3RL graph with no induced copy of $K_{1,r}$. This is simply because $K_r$ is an induced subgraph of $\overline{K_{1,r}}$, so $\overline{\mathcal{H}}$, the sequence of complements of the graphs $f(G_k^r)$ we have just constructed, does not contain any star of order $11$ or greater as an induced subgraph. Now let $J$ be any graph of order $n \geq R(10,10)$. By Ramsey's Theorem $J$ must contain $K_{10}$ or $\overline{K_{10}}$ as a subgraph. In the former case $\mathcal{H}$ is 3RL but without $J$ as an induced subgraph and in the latter case $\overline{\mathcal{H}}$ is 3RL but without $J$ as an induced subgraph. This concludes the proof of Theorem~\ref{thmoverkill}.

\section{A construction of high universality}\label{sechigh}

In this section we prove Theorem \ref{thmarbitrary}, to which end we shall use the following construction. 

Given vertex disjoint graphs $G_1 = (V_1, E_1)$ and $G_2 = (V_2, E_2)$, where $\left|V_1\right|=\left|V_2\right|$, we construct a graph $H = G_1 \oplus G_2$ by joining $G_1$ and $G_2$ via a random bipartite graph. Formally, $V(H) = V_1 \cup V_2$ and $E(H) = E_1 \cup E_2 \cup E_3$, where $E_3$ is formed by joining independently at random each pair $(x,y) \in V_1 \times V_2$ with probability $1/2$. 

%Let $V(G)$ be a disjoint union of $V_1,V_2,V_3$ and $V_4$, each of which is of size $n$. Put an induced copy of $K_{n,n}$ between $V_1$ and $V_2$ and its complement between $V_3$ and $V_4$. Between $V_1\cup V_2$ and $V_3\cup V_4$ we select each edge randomly and independently with probability $1/2$. The resulting graph $G$ is a.a.s. 3RL as $n$ tends to infinity, as we shall see below. 

%We would like to apply the above construction in a more general setting: put a graph $G_1$ on $V_1\cup V_2$, a graph $G_2$ on $V_3\cup V_4$ and a random graph between $V_1\cup V_2$ and $V_3\cup V_4$ as before. Call the resulting graph $H=G_1\oplus G_2$. The following two lemmas provide more information about the $3$-local profile of $H$. 

The special case of this construction in which $G_1 = K_{n,n}$ and $G_2 = \overline{K_{n,n}}$ was used by Chung, Graham and Wilson \cite{CGW} in order to demonstrate that a graph $H$ that behaves random-like with respect to all $3$-vertex subgraphs, that is, when $(p_0(H), p_1(H), p_2(H), p_3(H)) = (1/8, 3/8, 3/8, 1/8)$, need not be quasirandom. Note that, due to Corollary \ref{cor3rl}, being random-like with respect to all $3$-vertex subgraphs is equivalent to being 3RL. The following two lemmas provide more information on the $3$-local profile of $G_1 \oplus G_2$. 

\begin{lemma}\label{lemplus1}
$H = G_1 \oplus G_2$ is a.a.s. Goodman if and only if $G_1$ and $G_2$ are Goodman.
\end{lemma}

\begin{proof}
For every $0 \leq i \leq 3$ and sufficiently large $n$, the probabilities of a random vertex-triple of $V(H)$ having exactly $i$ vertices in $G_1$ are roughly binomially distributed. It thus follows by the definition of $G_1 \oplus G_2$ and by standard bounds on the tails of the binomial distribution that a.a.s.
\begin{equation}\label{eq: 6}
p_3(H) = \frac{1}{8}\left(p_3(G_1)+p_3(G_2)\right)+\frac{3}{8}\cdot \left(\frac{1}{2}\right)^2\cdot\left(p_e(G_1)+p_e(G_2)\right) \,,
\end{equation}
and similarly a.a.s. 

\begin{equation}\label{eq: 7}
p_0(H) = \frac{1}{8}\left(p_0(G_1)+p_0(G_2)\right)+\frac{3}{8}\cdot \left(\frac{1}{2}\right)^2\cdot\left(p_e(\overline{G_1})+p_e(\overline{G_2})\right) \,.
\end{equation}

Since $p_e(G)+p_e(\overline{G})=1$, adding the above equations we obtain 

$$
p_0(H)+p_3(H)=\frac{1}{8}\left[(p_0(G_1)+p_3(G_1))+(p_0(G_2)+p_3(G_2))\right]+\frac{3}{16} \,,
$$
which equals $1/4$ if and only if 

\begin{equation}\label{eq: 8}
\left(p_0(G_1)+p_3(G_1)\right)+\left(p_0(G_2)+p_3(G_2)\right)=\frac{1}{2} \,.
\end{equation}

By Goodman's bound,~\eqref{eq: 8} can only hold when $p_0(G_1)+p_3(G_1)=p_0(G_2)+p_3(G_2)=\frac{1}{4}$, that is, when both $G_1$ and $G_2$ are Goodman.
\end{proof}

\begin{lemma}\label{lemplus}
$H=G_1\oplus G_2$ is a.a.s. 3RL if and only if $G_1$ and $G_2$ are Goodman and $p_i(G_1)=p_{3-i}(G_2)$ for all $0\leq i \leq 3$.  
\end{lemma}

\begin{proof}
If $G_1$ and $G_2$ satisfy the conditions of the lemma, then $p_e(G_1)=p_e(G_2)=1/2$ and $p_3(G_1)+p_3(G_2)=p_3(G_1)+p_0(G_1)=1/4$. It follows from \eqref{eq: 6} that $p_3(H)=1/8$.
Similarly, the conditions of the lemma and \eqref{eq: 7} yield $p_0(H)=1/8$, whence we conclude that $H$ is 3RL. 

Conversely, if $H$ is 3RL, it is Goodman, so by Lemma \ref{lemplus1} $G_1$ and $G_2$ must be Goodman as well, in which case $p_e(G_1)=p_e(G_2)=1/2$, and the identities \eqref{eq: 6} and \eqref{eq: 7} transform into 

$$
\frac{1}{8}=p_3(H)=\frac{1}{8}\left(p_3(G_1)+p_3(G_2)\right)+\frac{3}{32}
$$

and 

$$
\frac{1}{8}=p_0(H)=\frac{1}{8}\left(p_0(G_1)+p_0(G_2)\right)+\frac{3}{32} \,.
$$

It follows that 
$$
p_0(G_1)+p_0(G_2)=p_3(G_1)+p_3(G_2)=\frac{1}{4}=p_0(G_1)+p_3(G_1)=p_0(G_2)+p_3(G_2).
$$
Thus $p_0(G_1)=p_3(G_2)$ and $p_3(G_1)=p_0(G_2)$. By Lemma \ref{lemgood} we also have $p_1(G_1)=p_2(G_2)$ and $p_2(G_1)=p_1(G_2)$.
\end{proof}

%For large $n$ the probabilities of a random triple in $H$ having $i$ vertices in $G_1$ are roughly binomially distributed. Therefore
%$$p_3(H)=\frac{1}{8}\left(p_3(G_1)+p_3(G_2)\right)+2\cdot \frac{3}{8}\cdot \left(\frac{1}{2}\right)^3=\frac{1}{8}\left(p_3(G_1)+p_3(G_2)\right)+\frac{3}{32},
%$$
%which equals $1/8$ if and only if $p_3(G_1)+p_3(G_2)=1/4$. Symmetrically, $p_0(H)=1/8$ if and only if $p_0(G_1)+p_0(G_2)=1/4$. To summarize, $H$ is 3RL if and only if 
%\begin{equation}\label{eq: 6}
%p_0(G_1)+p_0(G_2)=p_3(G_1)+p_3(G_2)=\frac{1}{4}.
%\end{equation} 
%
%This is easily seen to be true if $G_1$ and $G_2$ are Goodman with $p_i(G_1)=p_{3-i}(G_2)$ for all $0\leq i \leq 3$. On the other hand, if \eqref{eq: 6} holds, adding the equations we obtain 
%
%$$
%\left(p_0(G_1)+p_3(G_1)\right)+\left(p_0(G_2)+p_3(G_2)\right)=\frac{1}{2},
%$$
%
%which by Goodman's bound can only hold when $p_0(G_1)+p_3(G_1)=p_0(G_2)+p_3(G_2)=\frac{1}{4}$, that is when both $G_1$ and $G_2$ are Goodman, in which case in order to satisfy \eqref{eq: 6} we must have $p_0(G_1)=p_3(G_2)$ and $p_3(G_1)=p_0(G_2)$. By Lemma \ref{lemgood} we also have $p_1(G_1)=p_2(G_2)$ and $p_2(G_1)=p_1(G_2)$.  

Since two 3RL graphs satisfy the conditions of Lemma \ref{lemplus}, we obtain the following useful fact as an immediate corollary.

\begin{cor}\label{corplus}
If $G_1$ and $G_2$ are 3RL then $G_1\oplus G_2$ is a.a.s. 3RL.
\end{cor}

Lemma \ref{lemplus} and Corollary \ref{corplus} allow us to iterate the construction $G_1\oplus G_2$, taking the aforementioned example of Chung, Graham and Wilson as our starting point. Define $G_1 := K_{n,n} \oplus \overline{K_{n,n}}$ and having constructed $G_\ell$, define $G_{\ell+1} := G_\ell \oplus G_\ell$. Let $\mathcal{G}_\ell = (G_\ell)_{n=1}^{\infty}$, that is, we fix the $\ell$'s iteration and let $n$ go to infinity. By Lemma \ref{lemplus} the sequence $\mathcal{G}_1$ is a.a.s. 3RL and by Corollary \ref{corplus} it follows inductively that for each $\ell>1$ the sequence $\mathcal{G}_\ell$ is a.a.s. 3RL. 

Each graph $G_\ell \in \mathcal{G}_\ell$ consists of $2^\ell$ ``deterministic'' components, each of which is either a copy of $K_{n,n}$ or its complement. The edges connecting vertices from different components are picked independently at random with probability $1/2$ each. Now, if we select $2^\ell$ vertices from $G_\ell$ uniformly at random, the probability of choosing precisely one vertex from each deterministic component is a positive function of $\ell$. Since the obtained graph contains only randomly picked edges, the expected proportion of induced copies of \emph{any} graph $H$ of order $2^\ell$ is also a positive function of $\ell$. Hence, for a fixed $\ell$ and $n$ tending to infinity we will have $p_H(\mathcal{G}_\ell)>0$ for each such graph $G$.  

%Hence, for a fixed $\ell$ this will happen with a bounded away from $0$ as $n$ tends to infinity) the induced subgraph will consist only of randomly picked edges, which in turn implies that $p_G(\mathcal{H}_\ell)>0$ for each graph . 

On the other hand, any subgraph of $G_\ell$ of order $m = 24 \ell \cdot 2^\ell$, contains by the pigeonhole principle either a clique or an independent set of size $12 \ell$. So assuming that $H_\ell$ is $m$-universal, every graph of order $m$ must contain such a set. However, the well-known lower bound on Ramsey numbers (see e.g. \cite{Bela}) states that there exist graphs on $2^{12 \ell/2} > m$ vertices without a clique or an independent set of size $12 \ell$, a contradiction. Thus we conclude that $\mathcal{G}_\ell$ is not $m$-universal, thereby completing the proof of Theorem \ref{thmarbitrary}.   

\subsection*{Remark.}

The standard proof of the bound $R(k,k) > 2^{k/2}$ has stronger consequences. Namely, it shows that for $n=2^{k/2}$ with high probability a random graph on $n$ vertices does not contain an induced copy of $K_k$ or $\overline{K_k}$. Applying this fact to the proof of Theorem \ref{thmarbitrary} shows that the proportion of graphs of order $m$ contained in $\mathcal{H}_\ell$ vanishes as $\ell$ tends to infinity. In other words, not only is $\mathcal{H}_\ell$ not $m$-universal, it actually contains ``very few'' different induced subgraphs of order $m$.
 
\section{Discussion}\label{secdisc}

There are many intriguing open problems regarding random-like sequences and universality. Several of them, including some problems on universality of tournaments, can be found in \cite{Limo}.

%We refer the reader to a number of open problems about universality mentioned in \cite{Limo}, including universality of tournaments. 

It would be very interesting to generalise our results to $m$-random-like sequences, that is, to study universalities of sequences whose densities of $m$-cliques and $m$-anticliques is $2^{- \binom{m}{2}}$. However, this seems to be a much more difficult task, since for $m > 3$ we no longer have the analogue of Goodman's Theorem. In fact, in our proof of Theorem \ref{thm4univ} we made use of the ``lucky coincidence'' that the random-like number of triangles and anti-triangles is also the minimal possible. Disproving a conjecture of Erd\H{o}s \cite{ErdConj}, it was shown by Thomason \cite{Thomason} that this is no longer true for $m \geq 4$. It would therefore also be interesting to investigate universalities of graphs whose densities of cliques and anti-cliques are the \emph{smallest possible} rather than random-like. 

%Franek-R\"{o}dl?!

Since 3RL is not enough to ensure 5-universality, one might ask which stronger random-like properties suffice. We propose the following question.  
%For random-like sequences, in order to have monotonicity, it seems natural to ask about universality of graphs that are $m$-random-like for all $m$ \emph{up to} some number $M$. We conecture the following.

\begin{ques}\label{conjHT}
Is it true that every sequence $\mathcal{G}$ that is $m$-random-like for every $m\leq M$ is $M$-universal?
\end{ques}  

By Theorem \ref{thm4univ} the answer to Question \ref{conjHT} is affirmative for $M\leq 4$, so $M=5$ is the first open case. Note that the iterated blow-up construction used to prove that 3RL sequences are not necessarily $5$-universal is not 4RL and thus does not provide a negative answer to Question \ref{conjHT}.

%The $3$-local profile of $G$ was extensively studied in \cite{profiles}. 

Our proof of Theorem \ref{thm4univ} established the existence of every possible $4$-vertex induced subgraph $H$ in a 3RL sequence $\mathcal{G}$. As noted in the Introduction, it follows from the induced graph removal lemma that in fact the corresponding density $p_H(\mathcal{G})$ must be bounded away from $0$. It would be interesting to determine, for every 4-vertex graph $H$, the minimum density $p_H(\mathcal{G})$ over all 3RL sequences $\mathcal{G}$.

%What about universality of graphs that are random-like with respect to \emph{all} cliques and anti-cliques? -> Universal by Fox's remark.
%Say more about universality (Fox's remark), $3$-profiles, quasirandomness, induced graph removal lemma, flag algebras... 
 
\bibliography{densities}

\begin{thebibliography}{10}

\bibitem{AFKS}
N.~Alon, E.~Fischer, M.~Krivelevich, and M.~Szegedy.
\newblock Efficient testing of large graphs.
\newblock {\em Combinatorica}, 20(4):451--476, 2000.

\bibitem{alon:probabilistic}
N.~Alon and J.H. Spencer.
\newblock {\em The Probabilistic Method, 3rd Ed.}
\newblock Wiley, New York, 2008.

\bibitem{Bela}
B.~Bollob\'{a}s.
\newblock {\em Modern Graph Theory}.
\newblock Graduate texts in mathematics. Springer, Heidelberg, corrected
  edition, 1998.

\bibitem{CGW}
F.~R.~K. Chung, R.~L. Graham, and R.~M. Wilson.
\newblock Quasi-random graphs.
\newblock {\em Combinatorica}, 9(4):345--362, 1989.

\bibitem{ErdConj}
P.~Erd\H{o}s.
\newblock On the number of complete subgraphs contained in certain graphs.
\newblock {\em Magyar Tud. Akad. Mat. Kulat\'{o} Int. K\"{o}zl}, 7:459--464,
  1962.

\bibitem{EH}
P.~Erd\H{o}s and A.~Hajnal.
\newblock Ramsey-type theorems.
\newblock {\em Discrete Applied Mathematics}, 25(1–2):37 -- 52, 1989.

\bibitem{Victor}
V~Falgas-Ravry and E.R. Vaughan.
\newblock Applications of the semi-definite method to the {T}ur\'an density
  problem for 3-graphs.
\newblock {\em Combin. Probab. Comput.}, 22(1):21--54, 2013.

\bibitem{Goodman}
A.~W. Goodman.
\newblock On sets of acquaintances and strangers at any party.
\newblock {\em Amer. Math. Monthly}, 66:778--783, 1959.

\bibitem{Limo}
N.~Linial and A.~Morgenstern.
\newblock Graphs with few $3$-cliques and $3$-anticliques are $3$-universal.
\newblock {\em Preprint}, 2013.
\newblock http://arxiv.org/abs/1306.2020.

\bibitem{SmallRam}
S.~Radziszowski.
\newblock Small {R}amsey numbers.
\newblock {\em Electronic Jounal of Combinatorics}, DS1, 2011.

\bibitem{Seinsche}
D.~Seinsche.
\newblock On a property of the class of {$n$}-colorable graphs.
\newblock {\em J. Combinatorial Theory Ser. B}, 16:191--193, 1974.

\bibitem{Simonovits}
M.~Simonovits.
\newblock A method for solving extremal problems in graph theory, stability
  problems.
\newblock In {\em Theory of {G}raphs ({P}roc. {C}olloq., {T}ihany, 1966)},
  pages 279--319. Academic Press, New York, 1968.

\bibitem{Thomason2}
A.~Thomason.
\newblock Pseudorandom graphs.
\newblock In {\em Random graphs '85 ({P}ozna\'n, 1985)}, volume 144 of {\em
  North-Holland Math. Stud.}, pages 307--331. North-Holland, Amsterdam, 1987.

\bibitem{Thomason}
A.~Thomason.
\newblock A disproof of a conjecture of {E}rd{\H{o}}s in {R}amsey theory.
\newblock {\em J. London Math. Soc}, 39:246--255, 1989.

\end{thebibliography}
\bibliographystyle{plain}

\end{document}